\documentclass[reqno]{amsart}

\usepackage{amsmath,amsfonts,amssymb,amsthm,enumerate,tikz,comment,hyperref,float}
\usetikzlibrary{arrows,positioning,calc,intersections,through,decorations.markings,matrix}
\usepackage[margin=1.5in]{geometry}
\restylefloat{table}
\theoremstyle{plain}
\newtheorem{theorem}{Theorem}[section]
\newtheorem{lemma}{Lemma}[section]
\newtheorem{corollary}{Corollary}[section]
\newtheorem*{MT}{Main Theorem}

\theoremstyle{definition}

\newtheorem{proposition}{Proposition}[section]

\theoremstyle{remark}
\newtheorem{remark}{Remark}[section]

\DeclareMathOperator{\Id}{\mathrm{Id}}

\newcommand\FanoPlane[1][1cm]{%
\begin{tikzpicture}[
mydot/.style={
  draw,
  circle,
  fill=white,
  inner sep=5pt}
]
\draw
  (0,0) coordinate (A) --
  (#1,0) coordinate (B) -- ($ (A)!.5!(B) ! {sin(60)*2} ! 90:(B) $) coordinate (C) -- cycle;
\coordinate (O) at
  (barycentric cs:A=1,B=1,C=1);
\draw (O) circle [radius=#1*1.717/6];
\draw (C) -- ($ (A)!.5!(B) $) coordinate (LC); 
\draw (A) -- ($ (B)!.5!(C) $) coordinate (LA); 
\draw (B) -- ($ (C)!.5!(A) $) coordinate (LB); 

\foreach \Nodo in {B,C,O,LC,LA,LB}
  \node[mydot] at (\Nodo) {};    
	\node[mydot] at (A) {7,14};
	\node[mydot] at (C) {1,8};
	\node[mydot] at (B) {6,13};
	\node[mydot] at (LC) {4,11};
	\node[mydot] at (LA) {2,9};
	\node[mydot] at (LB) {5,12};
	\node[mydot] at (O) {3,10};
\end{tikzpicture}%
}

\begin{document}

\title{Symmetric Generation of \(M_{22}\)}

\author[Z. Hasan]{Zahid Hasan}
\address{ZH: Department of Mathematics \\ California State University at
San Bernardino, \newline San Bernardino, CA 92407, USA}
\email{zhasan@csusb.edu}
\author[B. Lim]{Bronson Lim}
\address{BL: Department of Mathematics \\ University of Oregon \\ Eugene,
OR 97403, USA}
\email{bcl@uoregon.edu}

\maketitle

\begin{abstract}
	We give a computer-free proof that the Mathieu group \(M_{22}\) is a
	homomorphic image of the progenitor \(2^{\star 14}:L_3(2)\) factorized by three
	relations.
\end{abstract}

\section{Introduction}
\label{sec:intro}

\subsection{Symmetric generation of Mathieu groups.}
\label{ssec:mathieu-grps}

We construct a symmetric presentation for the middle Mathieu group \(M_{22}\).
The history of \(M_{22}\), and its siblings, is of independent interest. To see
how this article fits into the timeline it is worthwhile to recall it.

The Mathieu groups, \(M_{11},M_{12},M_{22},M_{23},M_{24}\), comprise a family of
five sporadic simple groups. Their discovery by \'Emile Mathieu in 1861 and 1873,
\cite{mathieu-61}, \cite{mathieu-73}, caused an uproar with even their
existence being questioned. Since these original papers, there have been several
alternative constructions which we mention presently.

Witt, \cite{witt1}, constructed the Mathieu groups as successive transitive
extensions of linear groups. Specifically, he used \(L_3(4)\) for the large
Mathieu groups and \(L_2(9)\) for the small ones. From this description it is
relatively easy to check the Mathieu groups are all simple and one can compute
various properties such as their order. In a related way, Witt also proved the
Mathieu groups are automorphism groups of Steiner systems.

Curtis found that the groups \(M_{12}\) and \(M_{24}\) have a highly symmetric
generating set, \cite{curt-mathieu}. From this generating set Curtis showed
that these groups were homorphic images of certain semi-direct products called
progenitors, see Section \ref{ssec:progenitors}. The group presentations constructed
with symmetric generation are called symmetric presentations. Finding and using
symmetric presentations has proved fruitful when dealing with larger finite
groups, particularly those related to the larger sporadic groups.

For example, Curtis used symmetric generation to give a new proof of the
existence of the first Janko group \(J_1\), \cite{curt-janko}. Elements of the
Janko group are exhibited as a word of length no more than 4 followed by an
element of \(L_2(11)\). This is a feat in and of itself, given that the smallest
permutation representation of \(J_1\) is on 266 letters. Later Curtis and Hasan,
\cite{cur-has}, developed an algorithm for multiplying elements using the
symmetric presentation.

Even more impressive is the use of symmetric generation to represent elements of
the Conway group \(\cdot\)O as a string of 64 symbols, \cite{cur-fair-09}. This
is a vast improvement over the prior smallest matrix representation, which is 24
dimensional (so 576 symbols), or the smallest permutation representation, which
is on 195,560 points.

Although many symmetric presentations are known, finding them is not
straightforward. There are sporadic groups which currently have no known
symmetric presentation, \cite[Section 6]{fair-progress}. These
include the Thompson group, Baby Monster, and the Monster. A symmetric
presentation for the Mathieu groups \(M_{11},M_{12},M_{23}\), and \(M_{24}\) can
be found in \cite[Sections 1,2, and 7]{Curt2}. In the present article we construct the first
involutory symmetric presentation of \(M_{22}\) using the transitive 14 point
action of \(L_3(2)\) and in doing so we finish the Mathieu family.

\subsection{Outline of paper.} 
\label{ssec:paper-outline}

In Section \ref{sec:prelims} we review symmetric generation of groups and prove
the Mathieu group \(M_{22}\) is symmetrically generated. In Section
\ref{sec:immediate-rels} we show our group \(\mathcal{G}\) possesses a
(maximal) subgroup \(\mathcal{M}\cong 2^3:L_3(2)\) and collect relations. In
Section \ref{sec:DCE} we perform manual double coset enumeration of
\(\mathcal{G}\) into double cosets of the form \(\mathcal{M\omega N}\). This is
an extension of the technique developed in \cite{Curt2}, first appearing in
\cite{wied}. In Section \ref{sec:main-result} we prove the Main Theorem. In
Section \ref{sec:comparison} we discuss how our presentation relates to the
standard one, list the maximal subgroups, and discuss other homomorphic images.

\subsection*{Acknowledgements}

We thank the referee for his/her invaluable input on an earlier version of this
manuscript.

\section{Preliminaries on symmetric generation and \(M_{22}\)}
\label{sec:prelims}

Throughout the paper, we use the standard Atlas notation for finite groups and
related concepts as described in \cite{ATLAS}.

\subsection{Progenitors and symmetric generation.}
\label{ssec:progenitors}

For a detailed discussion regarding the history of progenitors see \cite{Curt2}.
We begin with the definition.

Denote by \(2^{\star n}\) the \(n\)th free product of the cyclic group, \(C_2\), with
itself, i.e. \(2^{\star n}:=C_2\star\cdots\star C_2\). Picking generators
\(t_1,\ldots,t_n\) for the \(n\) copies of \(C_2\), we can write \(2^{\star
n}:=\langle t_1\rangle\star \cdots \star \langle t_n\rangle\cong \langle
t_1,\ldots,t_n\mid t_i^2 = 1\text{ for }i = 1,\ldots,n\rangle\). 

The symmetric group, \(\Sigma_n\), acts naturally on \(2^{\star n}\) by permuting
the generators. Let \(\mathcal{N}\) be a transitive subgroup of \(\Sigma_n\),
which we call the \textit{control group}. Define \(\mathcal{P}\) to be the
associated split extension:
\[
	\mathcal{P} = 2^{\star n}:\mathcal{N}.
\]
Then \(\mathcal{P}\) is called a \textit{progenitor} and the elements \(t_i\) are called
\textit{symmetric generators} \cite{Curt2,wied}\footnote{There is a more general
notion of progenitor where one replaces \(C_2\) with an arbitrary finite group
\(H\), see \cite[Section III]{Curt2} or \cite{fair-progress}. We will not need
this here.}.

Any element of \(\mathcal{P}\) can be written in the form \(\pi\omega\), where
\(\pi\in\mathcal{N}\) and \(\omega\) is a word in the symmetric generators.  We
obtain a distinguished copy of \(\mathcal{N}\) in  \(\mathcal{P}\) by taking
\(\omega\) to be the empty word. We will also refer to this subgroup as the
control group.

An epimorphic image of \(\phi:\mathcal{P\to G}\) of a progenitor
\(\mathcal{P}\) is called \textit{symmetrically generated} if 
\begin{itemize}
	\item The restriction to the control subgroup is an isomorphism onto its
		image.
	\item The \(\phi(t_i)\), for \(i = 1,\ldots,n\), are distinct involutions.
	\item The set \(\{\phi(t_i)\mid i = 1,\ldots,n\}\) generate
		\(\mathcal{G}\): \(\mathcal{G}=\langle \phi(t_i)\rangle_{i=1,\ldots,n}\).
\end{itemize}
If \(\mathcal{G}\) is symmetrically generated, we will abuse notation and write
\(t_i\) for both the symmetric generator in \(\mathcal{P}\) and its homomorphic
image in \(\mathcal{G}\). All of our computations will take place inside
\(\mathcal{G}\) so this should not cause confusion.

What we have called a symmetrically generated group is not the always the
definition used. One does not always require the stronger condition that the
\(t_i\) generate \(\mathcal{G}\). There are examples, e.g. \cite[Table
3.7-9]{Curt2}, where one also needs \(\phi(\mathcal{N})\) as well. However, our
construction will satisfy this condition so we include it.

If \(\mathcal{G}\) is symmetrically generated, we say the set
\(\{t_i\mid i=1,\ldots,n\}\) is a \textit{symmetric generating set} and the
corresponding quotient 
\[
	\mathcal{P}/\ker(\phi)\cong\mathcal{G}
\]
a \textit{symmetric presentation} of \(\mathcal{G}\).

As a corollary of the Feit-Thompson odd order theorem, any finite non-Abelian
simple group is symmetrically generated \cite[Lemma 3.6]{Curt2}.

Symmetrically generated groups often arise by factoring progenitors by relations
between words in the symmetric generators and permutations. That is, if
\(\pi_1,\ldots,\pi_m\in\mathcal{N}\) and \(\omega_1,\ldots,\omega_m\) are words
in the \(t_i\)s, we define
\[
	\mathcal{G} = \frac{\mathcal{P}}{\pi_1\omega_1,\ldots,\pi_m\omega_m}
\]
to mean the quotient of \(\mathcal{P}\) be the normal closure of \(\langle
\pi_1\omega_1,\ldots,\pi_m\omega_m\rangle\) in \(\mathcal{P}\).

\subsection{Symmetric generation of \(M_{22}\).} 
\label{ssec:main-result}

The main theorem is motivated by the following proposition. The reader is
referred to \cite{ATLAS} for the permutation representations used in this
section.  We will look inside the group structure of \(M_{22}\) for a symmetric
generating set consisting of involutions with control group \(L_3(2)\).  

\begin{proposition}
	There exist a symmetric generating set \(T=\{t_1,...,t_{14}\}\) of \(M_{22}\),
	such that \(|t_i|=2\).
	\label{prop:main}
\end{proposition}

\begin{proof}
	Consider the permutation representation of \(M_{22}\) on 176 letters given by
	the action of \(M_{22}\) on the set of cosets of the maximal subgroup \(A_7\).
	Within the maximal subgroup \(2^3:L_3(2)\) there are three class of subgroups
	isomorphic to \(L_3(2)=\mathcal{N}\).  In one of these classes there exists a
	point stabilizer isomorphic to \(A_4\) such that \(C_{M_{22}}(A_4)\cong A_4\). 
	
	Take an element of order 2 in the centralizer, say \(t\in C_{M_{22}}(A_4)\).
	We compute \(\mathcal{N}^t=C_{\mathcal{N}}(t)=A_4\). Thus 
	\(|t^\mathcal{N}|=|\mathcal{N:N}^t|=|\mathcal{N}:A_4| = 14\), so \(t\) has 14
	conjugates under the action of \(\mathcal{N}\).  We can label these conjugates as
	\(t_1,...,t_{14}\) so that the generators \(x\) and \(y\) of \(L_3(2)\) act as
	\(x=(1,2,3,4,5,6,7)(8,9,10,11,12,13,14)\) and
	\(y=(1,12)(2,3)(4,11)(5,8)(6,13)(9,10)\) on \(\{t_1,...,t_{14}\}\).

	We must show the elements \(t_i\) generate \(M_{22}\). Define \(H=\langle
	t_1,...,t_{14}\rangle\).  Since \(L_3(2)\) permutes the \(t_i\) we have
	\(L_3(2)\leq N_{M_{22}}(H)\).  Furthermore, \(2^3:L_3(2)\cong\langle
	L_3(2),t_1t_8\rangle\) is also a subgroup of the normalizer \(N_{M_{22}}(H)\)
	since \(t_i\) normalize \(H\).
	But \(2^3:L_3(2)\neq N_{M_{22}}(H)\), since \(t_1\notin 2^3:L_3(2)\). This
	implies that the maximal subgroup \(2^3:L_3(2)\) is a proper subgroup of
	\(N_{M_{22}}(H)\). This can only happen if \(H\) is normal. As \(H\) is
	nontrivial, it must be the entire group.
\end{proof}

\begin{remark}
	In terms of the standard permutation representation of \(M_{22}\) on 22 letters
	we have
	\begin{align*}
		t&=(2, 10)(3, 11)(4, 19)(5, 22)(6, 18)(7, 14)(9, 15)(12, 13) \\
		x&=(1, 12, 14, 10, 8, 17, 15)(2, 18, 22, 13, 3, 7, 9)(5, 6, 19, 21, 20, 11, 16) \\
		y&=(2, 9)(3, 4)(5, 6)(7, 13)(10, 15)(11, 19)(12, 14)(18, 22)  \\
		\mathcal{N}&=\langle x,y\rangle.
	\end{align*}
	By Proposition \ref{prop:main} we have \(M_{22}=\langle \mathcal{N},t\rangle\).
	\label{rem:main}
\end{remark}

\subsection{The progenitor considered.}
In lieu of Proposition \ref{prop:main}, we construct the progenitor
\(2^{*14}:L_3(2)\) 
\[
	2^{*14}:L_3(2) = \langle x, y, t | x^7, y^2, (xy)^3, [x,y]^4, t^2, [t^{x^2},
	yx^{-1}], [t,y]\rangle.
\]
The permutation representation, \(L_3(2)\hookrightarrow \Sigma_{14}\), is
discussed in Section \ref{ssec:l32-action}. Our main result is:

\begin{MT}
	The Mathieu group \(M_{22}\) is a homomorphic image of the progenitor
	\(2^{*14}:L_3(2)\). In particular, the group
	\[
		\mathcal{G} =
		\frac{2^{\star 14}:L_3(2)}{(yt^{x^2})^5,(xyx^2t^x)^5,(xt)^8}
	\]
	is isomorphic to \(M_{22}\).
	\label{thm:main}
\end{MT}

We have the following direct corollary of Proposition \ref{prop:main} and 
Remark \ref{rem:main}:

\begin{corollary}
	\( |\mathcal{G}|\geq 443,520\).
	\label{cor:main-order}
\end{corollary}

\begin{proof}
	One computes that the relations defining \(\mathcal{G}\) are satisfied by the
	\(x,y,t\) defined in Remark \ref{rem:main}.
\end{proof}

\subsection{14 point action of $L_3(2)$.}
\label{ssec:l32-action}

Let \(\mathcal{N} = L_3(2)\) be the automorphism group of the Fano plane,
\(\mathbb{P}^2(\mathbb{F}_2)\). The permutation representation used in this
paper is not the standard one, the action on the 7 points of
\(\mathbb{P}(\mathbb{F}_2)\), as it is on 14 points.

The stabilizer of any of the seven points in the Fano plane has an index 2
subgroup isomorphic to \(A_4\). Taking the coset action of \(L_3(2)\) on \(A_4\)
gives the 14 point action considered here. Since the \(A_4\)s are not maximal,
this action is imprimitive with two blocks of size seven.

Label the points of the Fano plane \(1,\ldots,7\). In terms of the cosets above,
these points correspond to the isomorphic copy of \(A_4\subset\mathcal{N}^i\) in
the stabilizer of each point. We will view the other cosets in each point
stabilizer as the number \(i+7\), e.g. the non-\(A_4\) coset in the point
stabilizer \(\mathcal{N}^1\) is \(8\). 

The presentation \(L_3(2) = \langle x,y\mid x^7,y^2,(xy)^3,[x,y]^4\rangle\) is
standard. With our action, we may choose \(x\) to be the permutation
\[
	x\sim (1,2,3,4,5,6,7)(8,9,10,11,12,13,14)
\]
and it is possible to take the generating involution, \(y\), to be
\[
	y\sim (1,12)(2,3)(4,11)(5,8)(6,13)(9,10).
\]

In Figure \ref{fig:fano-plane}, we have labelled each vertex appropriately. This
labelling will be important later when we are proving relations.

This permutation representation is transitive. The point stabilizer of any of
the points, say \(7\), has three orbits:
\(\{\{7\},\{14\},\{1,\ldots,6,8\ldots,13\}\}\). If two points do not lie over
the same vertex, as in Figure \ref{fig:fano-plane}, then their two point stabilizer is
trivial. If they do lie over the same vertex, say \(\{1,8\}\), then the two
point stabilizer is isomorphic to \(A_4\).

\begin{figure}[h]
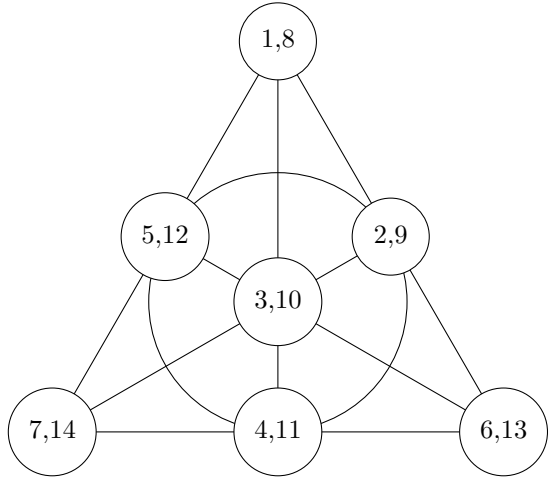

	\FanoPlane[6cm]
	\caption{The Fano Plane} 
	\label{fig:fano-plane}
\end{figure}

\subsection{Notation.} 
\label{ssec:notation}

We use the permutation action above and the association \(t = t_7\) to define
\(2^{\star 14}:L_3(2)\). Then the group in question is:
\[
	\mathcal{G} = \frac{2^{\star 14}:L_3(2)}{ (yt^{x^2})^5, (xyx^2t^x)^5, (xt)^8}.
\]

We make the following definitions:
\begin{align*}
	& s_i = t_it_{i+7}\text{ for }1\leq i\leq 14; \\
	& \mathcal{N} = \langle x,y\rangle\cong L_3(2); \\
	& \mathcal{M} = \langle x,y,s_i\rangle = \langle x,y,s_1\rangle.
\end{align*}
The subscript is understood to be taken modulo 14, i.e. \(t_{15} = t_1\).

We denote the \textit{point stabilizer} in \(\mathcal{N}\) of the points
\(i_1,...,i_j\) by \(\mathcal{N}^{i_1...i_j}\) and the \textit{coset
stabilizer} of \(\mathcal{M}t_{i_1},...,t_{i_j}\) by \(\mathcal{N}^{(i_1...i_j)}\).
Neccesarily we have \(\mathcal{N}^{i_1...i_j}\leq \mathcal{N}^{(i_1...i_j)}\) but
this containment is typically proper.  For two words \(\omega_1\) and
\(\omega_2\) in the \(t_i\)s we define the equivalence relation \(\omega_1\sim
\omega _2\) if \(\mathcal{M}\omega_1=\mathcal{M}\omega_2\).

\subsection{Outline of proof.} 
\label{ssec:proof-outline}

We show \(\mathcal{G}\) acts faithfully, transitively, and primitively on the
set of cosets \(\{\mathcal{M}\omega\}\). We then use Iwasawa's lemma,
\cite{Iw1,wilson-09}, to deduce \(\mathcal{G}\) is simple. We also show \(\mathcal{G}\)
has order 443,520. This is sufficient to conclude \(\mathcal{G}\cong
M_{22}\).

\section{\(\mathcal{M}\cong 2^3:L_3(2)\)}
\label{sec:immediate-rels}

We prove \(\mathcal{M}\cong 2^3:L_3(2)\). We do so by proving the \(s_i\)
generate a group isomorphic to \(2^3\). We also collect relations that will be
used when proving coset collapsing later on.

\subsection{Immediate relations.}

Recall our notational conventions: we write \(t_i\) for the homomorphic image of
\(t_i\) in \(\mathcal{G}\) and \(t=t_7\).

The relation \((yt^{x^2})^5=\text{Id}\) implies \(y=t_2t_3t_2t_3t_2\). This gives a family of
relations \(y=(t_2t_3t_2t_3t_2)^g\) where \(g\in C_\mathcal{N}(y)\) centralizes \(y\). The
complete list is
\begin{align*}
	y &= t_2t_3t_2t_3t_2 = t_1t_{12}t_1t_{12}t_1 = t_9t_{10}t_9t_{10}t_9=
	t_5t_8t_5t_8t_5 \\
	&= t_3t_2t_3t_2t_3 = t_{12}t_1t_{12}t_1t_{12} = t_{10}t_9t_{10}t_9t_{10} =
	t_8t_5t_8t_5t_8.
\end{align*}

Conjugating these relations gives us, for \(i\neq j\),
\begin{equation}
	\alpha_{i,j} = t_it_jt_it_jt_i
	\label{rel:alpha}
\end{equation}
for some involution \(\alpha_{i,j}\in \mathcal{N}\). For example, \(y =
\alpha_{2,3} = \alpha_{5,8} = \alpha_{9,10}=\alpha_{10,9}\). Notice also that
\(\alpha_{i,j} = \alpha_{j,i}\) for \(i\neq j\).

The relation \( (xyx^2 t^x)^5=\Id\) implies \(
(xyx^2)^{-1}=t_1t_{12}t_8t_5t_1\). The permutation \( (xyx^2)^{-1}\) is of order
4 and the centralizer in \(\mathcal{N}\) is the cyclic of order 4. Conjugation
then gives
\[
	(xyx^2)^{-1} = t_1t_{12}t_8t_5t_1=t_5t_1t_{12}t_8t_5 =t_8t_5t_1t_{12}t_8 
	= t_{12}t_8t_5t_1t_{12}.
\]

Acting by \(\mathcal{N}\) gives us a family of relations
\begin{equation}
	\beta_{i,j,\bar{i}} = t_it_jt_{\bar{i}}t_{\bar{j}}t_i =
	t_{\bar{j}}t_it_jt_{\bar{i}}t_{\bar{j}}
	\label{rel:beta}
\end{equation}
where \(\bar{i} \equiv_{14} i + 7\). For example, \( (xyx^2)^{-1}=\beta_{1,12,8} =
\beta_{5,1,12} = \beta_{8,5,1} = \beta_{12,8,5}\). 

\subsection{The isomorphism \(\mathcal{M}\cong 2^3:L_3(2)\)} 

To see the isomorphism between the subgroup \(\mathcal{M}\) and the holomorphism
group \(2^3:L_3(2)\), we only need to establish \(\langle s_i\rangle\cong 2^3\).
The semidirect product structure is then inhertied from \(\mathcal{G}\). Recall,
\(s_i = t_it_{i+7}\), i.e.  \(s_1 = t_1t_8\) and \(s_8 = t_8t_1\). 

\begin{lemma}
	The elements \(s_i\) are involutions and \(s_i = s_{i+7}\) for \(i =
	1,\ldots,7\).
	\label{lem:s-invols}
\end{lemma}

\begin{proof}
	It suffices to show \(s_8\) is an involution. We use Relation
	(\ref{rel:alpha}):
	\begin{align*}
	\alpha_{1,12}&=t_1t_{12}t_1t_{12}t_1 =t_1t_{12}t_8t_5t_5t_8t_1t_{12}t_1 \\
	&=(t_1t_{12}t_8)t_5t_5t_8t_1t_{12}t_1 =\beta_{1,12,8}t_1t_5t_8t_1t_{12}t_1
	\end{align*} 
	so that
	\(t_8t_1=\beta_{1,12,8}^{-1}\alpha_{1,12}t_5t_8t_1t_{12}\). We then have:
	\begin{align*}
		s_8^2&= t_8t_1t_8t_1=t_8t_1\beta_{1,12,8}^{-1}\alpha_{1,12}t_5t_8t_1t_{12} \\
		&= \beta_{1,12,8}^{-1}\alpha_{1,12}t_1(t_8t_5t_8)t_1t_{12}
		=\beta_{1,12,8}^{-1}\alpha_{1,12}t_1(\alpha_{8,5}t_8t_5)t_1t_{12} \\
		&=\beta_{1,12,8}^{-1}\alpha_{1,12}\alpha_{8,5}t_{12}t_8t_5t_1t_{12}
		=\beta_{1,12,8}^{-1}\alpha_{1,12}\alpha_{8,5}\beta_{12,8,5} =\Id. \\
  \end{align*}
\end{proof}

Note that in the proof of Lemma \ref{lem:s-invols}, the relation \(
t_8t_1=\beta_{1,12,8}^{-1}\alpha_{1,12}t_5t_8t_1t_{12}\) implies
\(t_5t_8t_1=\beta_{1,12,8}^{-1}\alpha_{1,12}t_8t_1t_{12}\) tells us how to move
the symmetric generators past the elements \(s_i\). We deduce the family of
relations:
\begin{equation}
	\delta_{i,j}s_it_j = t_{\bar{j}}s_i,
	\label{rel:delta}
\end{equation}
where again \(\bar{j}\equiv_{14}j+7\). For example, \(\delta_{1,12} =
\beta_{1,12,8}^{-1}\alpha_{1,12}\). The \(\delta_{i,j}\) are involutions and
\(\delta_{i,j}\) pointwise fixes \(\{j,\bar{j}\}\) and interchanges \(i\) with \(\bar{i}\). So if we
conjugate Relation (\ref{rel:delta}) by \(t_{\bar{j}}\) we get
\[
	\delta_{i,j}t_{\bar{j}} s_i = s_it_j
\]
which implies \(\delta_{i,j} = \delta_{i,\bar{j}}^{-1} =
\delta_{i,\bar{j}}\). Moreover, since \(s_i = s_{i+7}\), we have
\(\delta_{i,j} = \delta_{\bar{i},j}\). 

\begin{lemma}
	The group \(\langle s_i\rangle_{i=1,\ldots,7}\) is a commutative 2-group.
	\label{lem:s-comms}
\end{lemma}

\begin{proof}
	It is enough to show \(s_1s_7 = s_7s_1\). We compute
	\begin{align*}
		s_1s_7 &= t_1(t_8s_7) = t_1 \delta_{1,7} s_7t_1 \\
		&= \delta_{1,7} (t_1s_7)t_1 = \delta_{1,7}^2 s_7s_1 \\
		&= s_7s_1.
	\end{align*}
\end{proof}

\begin{lemma}
	Let \(i,j,k,l\) be four points in the Fano Plane (Figure \ref{fig:fano-plane})
	such that no three of them are collinear, then
	\[
		s_is_j = s_ks_l.
	\]
	\label{lem:s-prods}
\end{lemma}

\begin{proof}
	It is enough to show \(s_5s_7 = s_3s_4\). We compute \(s_4s_7t_2\) in two
	different ways:
	\begin{align*}
		s_4(s_7t_2) &= s_4\delta_{2,7}t_9s_7 = \delta_{2,7}s_5t_9s_7 \\
		&= \delta_{2,7}\delta_{2,5} t_2s_5s_7
	\end{align*}
	and
	\begin{align*}
		s_4s_7t_2 &= s_7(s_4t_2) = s_7\delta_{2,4} t_9s_4 \\
		&= \delta_{2,4}s_3t_9s_4 = \delta_{2,4}\delta_{2,3}t_2s_3s_4.
	\end{align*}
	Since \(\delta_{2,7}\delta_{2,5} = \delta_{2,4}\delta_{2,3}\) we are finished.
\end{proof}

\begin{lemma}
	If \(i,j,k\) are collinear in the Fano Plane (Figure \ref{fig:fano-plane}),
	then
	\[
		s_is_js_k = \Id.
	\]
	\label{lem:s-prod-rule}
\end{lemma}

\begin{proof}
	Since \(\mathcal{N}\) acts transitively on the set of lines in the Fano plane,
	it is enough to show \(s_1s_5s_7 = \Id\). We compute
	\begin{align*}
		s_1s_5(s_7t_4)&= s_1s_5(\delta_{4,7}t_{11}s_7) = \delta_{4,7}
		s_3(s_2t_{11})s_7 \\
		&= \delta_{4,7} s_3 \delta_{2,4}t_4s_2s_7 =
		\delta_{4,7}\delta_{2,4}\delta_{1,4}
		t_{11}s_1s_2s_7 \\
		&= t_{11}s_1s_2s_7.
	\end{align*}
	By Lemma \ref{lem:s-prods}, we know \(s_1s_2 = s_4s_7\). Thus
	\[
		s_1s_5s_7t_4 = t_{11}(s_1s_2)s_7 = t_{11}s_4s_7s_7 = t_4
	\]
	and we conclude \(s_1s_5s_7 = \Id\).
\end{proof}

The following corollary follows immediately from Lemma \ref{lem:s-prod-rule} and
Relation (\ref{rel:delta}).

\begin{corollary}
	For \(i,j,k\) collinear in the Fano Plane (Figure
	\ref{fig:fano-plane}), there exists an involution \(\sigma_{i,j}\) such that
	\[
		t_it_jt_k = \sigma_{i,j}t_it_jt_{k+7}.
	\]
	\label{cor:linearity-sigma}
\end{corollary}

The multiplication table for the \(s_i\) is given by Figure
\ref{fig:fano-plane}. For example, \(s_7s_2 = s_3\) as these points are
collinear. We conclude

\begin{theorem}
	\(\mathcal{M}\cong 2^3:L_3(2)\). 
	\label{thm:M-iso}
\end{theorem}

\begin{proof}
	Follows from Lemmas \ref{lem:s-comms} and \ref{lem:s-prod-rule}.
\end{proof}

\subsection{Two more families of relations.}

The relation \( (xt)^8\) can be written as \(t_7t_1t_2t_3 = xt_7t_6t_5t_4\)
Conjugating by \(\mathcal{N}\) gives us another family of relations:
\begin{equation}
	t_it_jt_kt_l = \epsilon_{ip}t_it_pt_qt_r,
	\label{rel:big}
\end{equation}
where the permutation \(\epsilon_{ip}\) is determined, as indicated, by
\(t_it_p\). This follows since \(t_p\neq t_{i+7}\) and so the two point
stabilizer is trivial. 

We devote the rest of this section to proving the existence of the following
relation
\begin{equation}
	\gamma_{i,j,k} = t_it_jt_kt_it_j,
	\label{rel:gamma}
\end{equation}
where \(i,j,k\) are distinct and collinear on the Fano plane, see Figure
\ref{fig:fano-plane}. We will need a lemma. 

\begin{lemma}
	There exists \(\sigma\in\mathcal{N}\) such that \(t_{14}t_6t_3t_{11} = \sigma
	t_5t_3t_8t_{14}\).
	\label{lem:key-lemma-reln4}
\end{lemma}

\begin{proof}
	We use Relation (\ref{rel:big}):
	\begin{align*}
		xt_{14}t_6 &= t_7t_1t_2t_3t_4t_5(t_6t_7
		t_{14})t_6 \\
		&=t_7t_1t_2t_3t_4t_5(\delta_{6,7}s_7s_6) \\
		&=\delta_{6,7}t_{14}t_9t_8t_5t_{11}(t_3s_4) \\
		&=\delta_{6,7}t_{14}t_9t_8t_5t_{11}(\delta_{3,4}s_4t_{10}) \\
		&=\delta_{6,7}\delta_{3,4}t_9t_{14}t_1t_{13}(t_4s_4)t_{10} \\
		&=\delta_{6,7}\delta_{3,4}t_9t_{14}t_1t_{13}t_{11}t_{10} \\
	\end{align*}
	and multiply both sides by \(t_3t_{11}\):
	\begin{align*}
		xt_{14}t_6t_3t_{11} &= \delta_{6,7}\delta_{3,4}
		t_9t_{14}t_1t_{13}(t_{11}t_{10}t_3)t_{11} \\
		&= \delta_{6,7}\delta_{3,4}
		t_9t_{14}t_1t_{13}(\delta_{3,4}s_3t_4)t_{11} \\
		&=\delta_{6,7} t_5t_{13}t_8(t_{14}s_1) \\
		&=\delta_{6,7} t_5t_{13}t_8\delta_{1,7}s_1t_7 \\
		&=\delta_{6,7}\delta_{1,7} t_{12}t_{11}t_8t_7 \\
	\end{align*}
	and now multiplying both sides by \(t_{14}t_8\):
	\begin{align*}
		xt_{14}t_6t_3t_{11}t_{14}t_8 &= \delta_{6,7}\delta_{1,7}
		t_{12}t_{11}(t_8t_7t_{14})t_8\\
		&= \delta_{6,7}\delta_{1,7}t_{12}t_{11}(\delta_{1,7}s_7s_1) \\
		&= \delta_{6,7} t_5(t_{10}s_5) \\
		&= \delta_{6,7} t_5(\delta_{3,5}s_5t_3) \\
		&=\delta_{6,7}\delta_{3,5} t_{12}s_5t_3 \\
		&= \delta_{6,7}\delta_{3,5}t_5t_3
	\end{align*}
	and this completes the proof.
\end{proof}

\begin{theorem}
	Relation (\ref{rel:gamma}) holds.
	\label{thm:relation-gamma}
\end{theorem}

\begin{proof}
	By Lemma \ref{lem:key-lemma-reln4}, we have:
	\begin{align*}
		t_{14}t_9t_8t_5 & = \sigma t_5t_3t_8t_{14} \\
		t_9t_{14}t_9t_8t_5 &= \sigma t_3t_8t_{14} \\
		\alpha_{9,14}t_9t_{14}t_8t_5 &= \sigma t_3t_8t_{14} \\
		t_9t_{14}t_8t_5 &=\alpha_{9,14}\sigma t_3t_8t_{14} \\
		t_{14}t_8t_5 &= \alpha_{9,14}\sigma t_8t_{14}.
	\end{align*}
	Since \(t_{14}t_8t_5 = \sigma_{14,8}t_{14}t_8t_{12}\), relation
	\ref{rel:gamma} now follows by conjugation.
\end{proof}

\section{Double Coset Enumeration Over \(\mathcal{M}\)}
\label{sec:DCE}

Double coset enumeration is the main tool for working with symmetric
presentations of finite groups. An introduction can be found in \cite{Curt2},
see also the papers \cite{BC2}, \cite{BC3}, \cite{fair-coxeter}. The primary
purpose of double coset enumeration, at least in this paper, will be to find an
upper bound on the size of \(\mathcal{G}\). A secondary, but not always useful,
purpose is that the enumeration process also furnishes a (collapsed) Cayley
Graph on which our group acts. This graph sometimes contains useful information
about the group, e.g. how the symmetric generators act on the cosets. In this
paper, we use the Cayley Graph will be used to show \(\mathcal{G}\) is simple.

One typically enumerates the double cosets of the form
\(\mathcal{N}\omega\mathcal{N}\), where \(\omega\) a word in the \(t_i\)s. This
paper will enumerate double cosets of the form \(\mathcal{M}\omega\mathcal{N}\).
Although \(\mathcal{M}\) is only eight times larger than \(\mathcal{N}\), the
work becomes substantially easier by replacing \(\mathcal{N}\) with
\(\mathcal{M}\) in the enumeration. Using a larger group is not uncommon, see
\cite{bray-curt-10} and \cite{wied}.

Recall, for two words \(\omega_1\) and \(\omega_2\) in the \(t_i\)s, we define
\(\omega_1\sim \omega_2\) if and only if \(\mathcal{M}\omega_1 =
\mathcal{M}\omega_2\). Denote by \([\omega]\) the double coset
\(\mathcal{M}\omega\mathcal{N}\), i.e. the orbit of \(\mathcal{M}\omega\) under
the right action of \(\mathcal{N}\). For example, we write \([t_7] =
\mathcal{M}t_7\mathcal{N}\).

Unfortunately, this subject is naturally computationally intensive. We will use
relations frequently and it can get very confusing. With this in mind, we
subscript the \(\sim\) symbol when using relations. This seems to be the
simplest way to be clear about when we are using which relation without
further cluttering the computation.

For example
\[
	t_7t_1t_8\sim_{\ref{rel:delta}} t_1t_8t_{14}=s_1t_{14}\sim t_{14}.
\]
Here we have used Relation (\ref{rel:delta}). The last \(\sim\) follows by the
definition of \(\mathcal{M}\). We will freely absorb the elements \(s_i\) into
\(\mathcal{M}\) without mention.

\subsection{Cosets of length 1 and 2.}
\label{ssec:cosets1-2}

Let \([\star]\) denote the coset \(\mathcal{M}\). Since \(\mathcal{N}\) is
transitive, there is a single orbit. We choose 7 as the orbit representative
giving us a new double coset \([t_7]\).

\subsubsection{\([t_7]\)}
\label{sssec:7}

Since \(s_7=t_7t_{14}\in \mathcal{M}\), we have \(t_7\sim t_{14}\) and so 
\begin{equation}
	\mathcal{N}^{(7)}\geq\langle \mathcal{N}^7,(2, 13)(3, 4)(5, 12)(6, 9)
	(7, 14)(10, 11)\rangle.
	\label{rel:stab7}
\end{equation}

The two orbits of \(N^{(7)}\) are \(\{7,14\}\) and \(\{1,...,6,8,...,13\}\). We
shoose the orbit representatives 7 and 1.

The \(\{7,14\}\) orbit goes back and the remaining orbit gives a new double
coset \([t_7t_1]\). 

\subsubsection{\([t_7t_1]\)}
\label{sssec:71}

Using Relation (\ref{rel:delta}) we have
\[
	t_7t_1t_8\sim_{\ref{rel:delta}} \delta_{1,8,14}s_8t_{14} \sim t_{14} \sim t_7.
\]
Thus \(t_7t_1\sim t_7t_8\) and we already have \(t_7t_1\sim t_{14}t_1\) so the
coset stabilizer expands
\begin{align}
	\label{rel:stab71}
	N^{(7,1)}&\geq\langle (2, 13)(3, 4)(5, 12)(6, 9)(7, 14)(10, 11), \\
	&(1, 8)(2, 10)(3, 9)(4, 6)(5, 12)(11, 13)\rangle.\nonumber
\end{align}

Since \(|\mathcal{N}^{(7,1)}|\geq 4\), there are at most \(42\) cosets of length
2. The 5 orbits of \(\mathcal{N}^{(71)}\) are \(\{\{ 1, 8 \},\ \{ 5, 12 \},\ \{
7, 14 \},\ \{ 2, 10, 11, 13 \},\ \{ 3, 4, 6, 9 \}\}\). We choose the orbit
representatives 1, 2, 3, 7, 12.

The \(\{1,8\}\) orbit evidently goes back. Relation (\ref{rel:alpha}) gives
\(t_7t_1t_7\sim_{\ref{rel:alpha}} t_7t_1\in [t_7t_1]\). Relation
(\ref{rel:gamma}) gives \(t_7t_1t_{12}\sim_{\ref{rel:gamma}} t_1t_7\in [t_7t_1]\).
The remaining two orbits give the two new double cosets \([t_7t_1t_2]\) and
\([t_7t_1t_3]\).

\subsection{Cosets of length 3.}
\label{ssec:cosets-3}
We consider the extensions of \(\mathcal{M}t_7t_1t_2\) and \(\mathcal{M}t_7t_1t_3\).

\subsubsection{\([t_7t_1t_2]\):}
\label{sssec:712}

In \([t_7t_1t_2]\), we have \(t_7t_1t_2\sim t_5t_1t_9\). Indeed, using Relation
(\ref{rel:delta}):
\[
	t_7(t_1s_2)\sim_{\ref{rel:delta}} (t_7\delta_{1,2})s_2t_8 \sim (t_5s_2)t_8
	\sim_{\ref{rel:delta}} t_{12}t_1 \sim_{\ref{rel:stab7}} t_5t_1.
\]
Thus the coset stabilizer is nontrivial:
\begin{align}
	\mathcal{N}^{(7,1,2)}&\geq\langle (2,9)(3, 11)(4, 10)(5, 7)(6, 13)(12, 14)\rangle.
	\label{rel:stab712}
\end{align}

Since \(|\mathcal{N}^{(7,1,2)}|\geq 2\), the number of elements in \([t_7t_1t_2]\)
is at most 84. The orbits of \(\mathcal{N}^{(7,1,2)}\) are: \(\{\{ 1 \},\{ 8 \},
\{ 2, 9 \}, \{ 3, 11 \}, \{ 4, 10 \}, \{ 5, 7 \}, \{ 6, 13 \}, \{ 12, 14\}\}\).
We use the following orbit representatives: 1, 2, 3, 4, 5, 6, 8, 12:

\begin{enumerate}
	\item[\(t_1\):] Using Relation (\ref{rel:alpha}) we have:
		\[
			t_7t_1t_2t_1\sim_{\ref{rel:alpha}} t_7\alpha_{1,2}t_1t_2\sim
			t_{11}t_1t_2\sim_{\ref{rel:stab7}} t_4t_1t_2\in [t_7t_1t_3].
		\]
	\item[\(t_2\):] It goes back: \(t_7t_1t_2t_2 = t_7t_1\in [t_7t_1]\).
	\item[\(t_5\):] We claim \(t_7t_1t_2t_5\sim t_8t_5t_{11}t_3\in
		[t_7t_1t_2t_3]\). We compute
		\begin{align*}
			(t_7t_1t_2)t_5t_3t_{11}t_5t_8 &\sim_{\ref{rel:big}}
			t_7t_6t_5t_4(t_3t_5t_3)t_{11}t_5t_8 \\
			&\sim_{\ref{rel:alpha}} t_7t_6t_5t_4\alpha_{3,5}t_3t_5t_{11}t_5t_8 \\
			&\sim t_{14}t_6(t_3t_{11}t_3)t_5t_{11}t_5t_8 \\
			&\sim_{\ref{rel:alpha}} t_{14}t_6\alpha_{3,11}t_3t_{11}t_5t_{11}t_5t_8  \\
			&\sim_{\ref{rel:alpha}} t_{12}t_{13}t_3(t_{11}t_5t_{11}t_5)t_8 \\
			&\sim t_{12}t_{13}t_3\alpha_{5,11}t_{11}t_8 \\
			&\sim t_4(t_8t_{10}t_{11})t_8 \\
			&\sim_{\ref{rel:gamma}} t_4\gamma_{8,10,11}t_{10} \\
			&\sim t_3t_{10}\sim \star.
		\end{align*}
		Multiplying on the right by \(t_8t_5t_{11}t_3\) completes the claim.
	\item[\(t_6\):] Using Relation (\ref{rel:gamma}) we have: 
		\[
			t_7(t_1t_2t_6)\sim_{\ref{rel:gamma}} t_7\gamma_{1,2,6} t_2t_1\sim
			t_4t_2t_1 \sim t_{11}t_2t_1\in [t_7t_1t_2].
		\]
	\item[\(t_8\):] Using Relation (\ref{rel:beta}) we have: 
		\[
			t_7(t_1t_2t_8)\sim_{\ref{rel:beta}} t_7\beta_{1,2,8}t_1t_9 \sim
			t_{12}t_1t_9\sim_{\ref{rel:stab7}} t_5t_1t_9\in [t_7t_1t_2].
		\]
	\item[\(t_{12}\):] We show \(t_7t_1t_2t_{12}\in [t_7t_1t_3t_7]\). When we discuss
		\([t_7t_1t_2t_3]\) in
		Section \ref{sssec:7123}, we show \([t_7t_1t_3t_7]=[t_7t_1t_2t_3]\). In Section
		\ref{sssec:713}, we show \(t_7t_1t_3\sim t_{12}t_1t_{10}\). Hence,
		\(t_{12}t_1t_{10}t_7\in [t_7t_1t_2t_3]\) and so it is enough to prove
		\(t_7t_1t_2t_{12}\in[t_{12}t_1t_{10}t_7]\). We compute:
		\begin{align*}
			(t_7t_1t_2)t_{12} &\sim_{\ref{rel:big}} (t_7t_6)t_5t_4t_3t_{12}
			\sim_{\ref{rel:gamma}} t_6t_7(t_4t_5t_4)t_3t_{12} \\
			&\sim_{\ref{rel:alpha}} t_6t_7\alpha_{4,5}t_4(t_5t_3t_{12}) \sim
			t_{13}t_{10}t_4\beta_{5,3,12}t_5t_{10} \\
			&\sim t_6(t_5t_8t_5)t_{10} 
			\sim_{\ref{rel:alpha}}t_6\alpha_{5,8}t_5t_8t_{10} \\
			&\sim t_{13}t_5t_8t_{10}\in [t_{12}t_1t_{10}t_7].
		\end{align*}
\end{enumerate}

The remaining orbits yield two new double cosets \([t_7t_1t_2t_3]\) and
\([t_7t_1t_2t_4]\).

\subsubsection{\([t_7t_1t_3]\)}
\label{sssec:713}

The proof of the relation \(t_7t_1t_3\sim t_{12}t_1t_{10}\) is analgous to the relation
\(t_7t_1t_2\sim t_5t_1t_9\) proved in Section 3.2.1. We conclude
\begin{align}
	\label{rel:stab713}
	\mathcal{N}^{(7,1,3)}&\geq\langle (2, 6)(3, 10)(4, 11)(5, 14)(7, 12)(9, 13)\rangle.
\end{align}

Since \(|\mathcal{N}^{(7,1,3)}|\geq 2\), the number of elements in \([t_7t_1t_3]\) is at most
84.  

The orbits of \(\mathcal{N}^{(7,1,3)}\) are: \( \{1\}, \{8\}, \{2,6\}, \{3,10\}, \{4,11\},
\{5,14\}, \{7,12\}, \{9,13\}\). We use the orbit representatives: 1, 2, 3, 4, 5,
7, 8, 9. 

\begin{enumerate}
	\item[\(t_1\):] Using Relation (\ref{rel:alpha}) we have: 
		\[ 
			t_7(t_1t_3t_1)\sim_{\ref{rel:alpha}} t_7 \alpha_{1,3}t_1t_3 \sim
			t_{13}t_1t_3\sim_{\ref{rel:stab7}} t_6t_1t_3\in[t_7t_1t_3].
		\]
	\item[\(t_2\):] We have
		\begin{align*}
			t_7t_1t_3t_2 &\sim_{\ref{rel:stab7}} (t_{14}t_1)t_3t_2 
			\sim_{\ref{rel:beta},\ref{rel:stab7}} t_{14}t_1(t_7t_3t_2) \\
			&\sim_{\ref{rel:gamma}} t_{14}t_1 \gamma_{7,3,2}t_3t_7 
			\sim t_9t_6t_3t_7 \\
			&\sim_{\ref{rel:stab7}} t_2t_6t_3t_7\in [t_7t_1t_2t_3].
		\end{align*}
	\item[\(t_3\):] It goes back: \(t_7t_1t_3t_3 = t_7t_1\in [t_7t_1]\).
	\item[\(t_4\):] We have
		\begin{align*}
			(t_7t_1t_3)t_4 &\sim_{\ref{rel:stab713},\ref{rel:stab7}}
			t_5(t_1t_{10}t_4)\sim_{\ref{rel:gamma}} t_5\gamma_{1,10,4}t_{10}t_1 \\
			&\sim t_9t_{10}t_1 \in [t_7t_1t_3].
		\end{align*}
	\item[\(t_5\):] We show \(t_7t_1t_3t_5\sim t_2t_4t_6t_1\in [t_7t_1t_2t_3]\). By conjugating 
		relation (\ref{rel:big}) we get:
		\begin{equation}
			\label{rel:big1}
			\sigma t_1t_3t_5t_7 = t_1t_6t_4t_2
		\end{equation}
		where \(\sigma =(1,6,4,2,7,5,3)(8,13,11,9,13,12,10)\). Hence,
		\begin{align*}
			t_7t_1t_3t_5(t_1t_6t_4t_2) &\sim_{\ref{rel:big1}} t_7t_1t_3t_5 \sigma
			t_1t_3t_5t_7 \\
			&\sim t_5t_6(t_1t_3t_1t_3)t_5t_7 \\ 
			&\sim_{\ref{rel:alpha}} t_5t_6\alpha_{1,3}t_1t_5t_7 \\
			&\sim (t_{12}t_{14})t_1t_5t_7 \\
			&\sim_{\ref{rel:stab71}} (t_5t_7t_1t_5t_7)\\
			&\sim_{\ref{rel:gamma}} \star.
		\end{align*}

	\item[\(t_7\):] We compute 
		\begin{align*}
			(t_7t_1)t_3t_7&\sim_{\ref{rel:beta}} t_7t_8t_{14}(t_3t_7)
			\sim_{\ref{rel:beta}} t_7t_8\beta_{14,3,7}t_{14}t_{10} \\
			&\sim (t_{10}t_{11})t_{14}t_{10}
			\sim_{\ref{rel:alpha}} t_{10}t_{11}(t_{10}t_{14}t_{10}) \\
			&\sim_{\ref{rel:alpha}} t_{10}t_{11}\alpha_{14,10}t_{10}t_{14}
			\sim (t_{14}t_4t_{10})t_{14} \\
			&\sim_{\ref{rel:stab712}} t_{12}t_{10}t_4t_{14}.
		\end{align*}
		The last \(\sim\) follows as \( t_{14}t_4t_{10}\in[t_7t_1t_2]\) and so we
		can use (\ref{rel:stab712}) to see \(t_{14}t_4t_{10}\sim
		t_{12}t_{10}t_{14}\). Finally, we have:
		\begin{align*}
			t_{12}t_{10}(t_4s_7)&\sim_{\ref{rel:delta}} t_{12}t_{10}
			\delta_{4,14,7}s_7t_{11}
			\sim t_9(t_1s_7)t_{11} \\
			&\sim_{\ref{rel:delta}} t_{12}\delta_{1,7}s_7t_8t_{11}
			\sim (t_6s_7)t_8t_{11} \\
			&\sim_{\ref{rel:delta}} (t_{13}t_8)t_{11} 
			\sim_{\ref{rel:stab71}} t_{13}t_1t_{11}\in[t_7t_1t_2].
		\end{align*}
		Hence, \(t_7t_1t_3t_7\sim t_{12}t_{10}t_4t_{14} \sim t_{13}t_1t_{11}t_7\in
		[t_7t_1t_2t_3]\) as desired.
	\item[\(t_8\):] Lastly, we compute: 
		\[
			(t_7t_1)t_3t_8\sim_{\ref{rel:stab71}} t_7(t_8t_3t_8)\sim_{\ref{rel:alpha}}
			t_7\alpha_{8,3}t_8t_3\sim t_7t_1t_3\in[t_7t_1t_3].
		\]

\end{enumerate}

The remaining orbit gives a new double coset \([t_7t_1t_3t_9]\).

\subsection{Cosets of length 4.}
\label{ssec:cosets-4}

We consider the extensions of \(\mathcal{M}t_7t_1t_2t_3\),
\(\mathcal{M}t_7t_1t_2t_4, \mathcal{M} t_7 t_1 t_3 t_9\).

\subsubsection{\([t_7t_1t_2t_3]\):}
\label{sssec:7123}

We have the relation \(t_7t_1t_2t_3 \sim t_{14}t_6t_2t_{10}\). This is proved
using the same technique as the relation \(t_7t_1t_2 \sim t_5t_1t_9\) from
Section \ref{sssec:712} and so is omitted. The relation expands the coset stabilizer

\begin{align}
	\mathcal{N}^{(7,1,2,3)}&\geq\langle (1, 6)(3, 10)(4, 12)(5, 11)(7, 14)(8, 13)\rangle.
	\label{rel:stab7123}
\end{align}

Since \(|\mathcal{N}^{(7,1,2,3)}|\geq 2\), the number of elements in \([t_7t_1t_2t_3]\)
is at most 84. The orbits of \(\mathcal{N}^{(7,1,2,3)}\) are: \(\{2\}, \{9\},
\{1,6\}, \{3,10\}, \{4,12\}, \{5,11\}, \{7,14\}, \{8,13\}\). We choose the orbit
representatives: 2, 3, 4, 5, 6, 7, 8, 9:

\begin{enumerate}
	\item[\(t_2\):]	Using Relation (\ref{rel:alpha}) we have: 
		\[
			t_7t_1(t_2t_3t_2) \sim_{\ref{rel:alpha}} t_7t_1\alpha_{2,3} t_2t_3 \sim
			t_7t_{12}t_2t_3.
		\]
		Multiplying on the right by \(t_7\) gives
		\begin{align*}
			(t_7t_1t_2t_3t_2)t_7 \sim t_7t_{12}(t_2t_3t_7)\sim_{\ref{rel:gamma}}
			t_7t_{12}\gamma_{2,3,7}t_3t_2\sim (t_3t_{13}t_3)t_2\sim_{\ref{rel:alpha}}
			t_3t_{13}t_2.
		\end{align*}
		Hence,
		\[
			t_7t_1t_2t_3t_2\sim t_3t_{13}t_2t_7\in[t_7t_1t_2t_3].
		\]
	\item[\(t_3\):] It goes back: \(t_7t_1t_2t_3t_3 = t_7t_1t_2\in[t_7t_1t_2]\).
	\item[\(t_4\):] Using Relation (\ref{rel:big}): 
		\[
			(t_7t_1t_2t_3t_4)\sim_{\ref{rel:big}}
			t_7t_6t_5\sim_{\ref{rel:stab71}} t_7t_{13}t_{15} \in [t_7t_1t_2].
		\]
	\item[\(t_5\):] Using Relation (\ref{rel:big}): 
		\[
			(t_7t_1t_2t_3)t_5\sim_{\ref{rel:big}} t_7t_6(t_5t_4t_5)
			\sim_{\ref{rel:alpha}} t_7t_6\alpha_{5,4}t_5t_4 \sim t_{10}t_{13}t_5t_4.
		\]
		Now if we multiply on the right by \(t_4t_3t_6\) we get
		\begin{align*}
			(t_7t_1t_2t_3t_5)t_4t_3t_6 &\sim t_{10}t_{13}t_5t_4t_4t_3t_6 
			\sim t_{10}t_{13}(t_5t_3t_6)\\
			&\sim_{\ref{rel:gamma}} t_{10}t_{13}\gamma_{5,3,6}t_3t_5
			\sim t_{12}t_{10}t_3t_5 \\
			&\sim t_{12} s_3 t_5\sim_{\ref{rel:beta}} \star.
		\end{align*}
		Hence, \(t_7t_1t_2t_3t_5\sim t_6t_3t_4\in[t_7t_1t_3]\).
	\item[\(t_6\):] We have just seen that \(t_7t_1t_2t_3t_5\sim t_6t_3t_4\) and
		so \(t_7t_1t_2t_3t_6\sim t_6t_3t_4t_5t_6\). By Relation (\ref{rel:big})
		\(t_3t_4t_5t_6 = xt_3t_2t_1t_7\) which gives
		\(t_6(t_3t_4t_5t_6)\sim t_6xt_3t_2t_1t_7\sim t_7t_3t_2t_1t_7\). Now we have:
		\begin{align*}
			t_7t_1t_2t_3t_6&\sim (t_7t_3t_2)t_1t_7\sim_{\ref{rel:gamma}}
			t_3(t_7t_1t_7)\sim_{\ref{rel:alpha}} t_3\alpha_{7,1}t_7t_1 \\
			&\sim t_{13}t_7t_1\sim_{\ref{rel:stab7}} t_6t_7t_1\in [t_7t_1t_2].
		\end{align*}
	\item[\(t_7\):]	We compute:
		\begin{align*}
			(t_7t_1t_2t_3)t_7 &\sim_{\ref{rel:stab7123}} t_{14}t_6(t_2t_{10}t_7)
			\sim_{\ref{rel:beta}} t_{14}t_6 \beta_{2,10,7} t_{10}t_2 \\
			&\sim t_3t_5t_{10}t_2
			\sim t_{12}(t_5t_3t_5)t_{10}t_2 \\
			&\sim_{\ref{rel:alpha}} t_{12} \alpha_{5,3}5(t_3t_{10})t_2 
			\sim t_{10}(t_5s_3)t_2 \\
			&\sim_{\ref{rel:delta}} t_{10}\delta_{5,3}s_3t_{12}t_2
			\sim (t_3s_3)t_{12}t_2 \\
			&\sim_{\ref{rel:delta}} t_{10}t_{12}t_2 
			\sim_{\ref{rel:stab71}} t_3t_5t_2\in [t_7t_1t_3]
		\end{align*}
	\item[\(t_8\):] We proved \(t_7t_1t_2t_3t_6\sim t_6t_7t_1\) above. Hence,
		\begin{align*}
			t_7t_1t_2t_3t_8&\sim (t_7t_1t_2t_3t_6)t_6t_8 
			\sim t_6t_7(t_1t_6t_8)\\
			&\sim_{\ref{rel:beta}} t_6t_7\beta_{1,6,8}t_1t_{13}
			\sim_{\ref{rel:beta}} (t_8t_{11}t_1)t_{13} \\
			&\sim t_8t_4t_{13}\in[7,1,3].
		\end{align*}
	\item[\(t_9\):] Using Relation (\ref{rel:beta}): 
		\begin{align*}
			t_7t_1(t_2t_3t_9)&\sim_{\ref{rel:beta}} t_7t_1\beta_{2,3,9}t_2t_{10} \sim
			t_{14}t_{13}t_2t_{10} \\
			&\sim_{\ref{rel:stab71}} t_{14}t_6t_2t_{10} \in
			[t_7t_1t_2t_3].
		\end{align*}
\end{enumerate}

This completes the coset \([t_7t_1t_2t_3]\). There are no new cosets.

\subsubsection{\([t_7t_1t_3t_9]\)}
\label{sssec:7139}

We first prove \(t_7t_1t_3t_9\sim t_{14}t_{11}t_3t_2\):
\begin{align*}
	t_7t_1(t_3s_2)t_3&\sim_{\ref{rel:delta}} t_7t_1\delta_{3,9}s_2s_3 \sim
	t_{14}(t_{11} s_7) \\
	&\sim_{\ref{rel:delta}} t_{14}\delta_{11,7}s_7t_4\sim t_7s_7t_4\\
	&\sim t_{14}t_4\sim_{\ref{rel:stab7}} t_{14}t_{11}.
\end{align*}
This adds the permutation
\[
	(1, 11)(2, 9)(4, 8)(5, 6)(7, 14)(12, 13)
\]
to the coset stabilizer.

We also have the relation \(t_7t_1t_3t_9\sim t_{9}t_{12}t_{10}t_1\). Indeed, we compute:
\begin{align*}
	(t_7t_1)t_3t_9 &\sim_{\ref{rel:big}} t_7t_6t_5t_4(t_3t_2t_3)t_9 \sim_{\ref{rel:alpha}}
	t_7t_6t_5t_4\alpha_{2,3}(t_3t_2t_9) \\
	&\sim_{\ref{rel:gamma}} t_7t_{13}\gamma_{8,11,3} t_{11}t_8s_2 \sim
	t_2t_{13}t_{11}t_8s_2. \\
	&\sim_{\ref{rel:delta}} t_2t_{13}t_{11}\delta_{8,2}s_2t_1 \sim
	t_9t_6(t_3s_2)t_1 \\
	&\sim_{\ref{rel:delta}} t_9t_6\delta_{3,2}s_2t_{10}t_1 \sim
	t_2(t_5s_2)t_{10}t_1\\
	&\sim_{\ref{rel:delta}} t_2\delta_{5,2}s_2t_{12}t_{10}t_1 \\
	&\sim t_9s_2t_{12}t_{10}t_1 \sim t_9t_{12}t_{10}t_1.
\end{align*}
But we already have \(t_7t_1t_3t_9\sim t_{12}t_1t_{10}t_9\) as \(t_7t_1t_3\sim
t_{12}t_1t_{10}\) and so we
add
\[
	(1,12,9)(2,8,5)(4,13,14)(6,7,11)
\]
to the coset stabilizer.

We conclude
\begin{align}
	\label{rel:stab7139}
	\mathcal{N}^{(7,1,3,9)}&\geq \langle (1, 11)(2, 9)(4, 8)(5, 6)(7, 14)(12, 13), \\
	&(1,12,9)(2,8,5)(4,13,14)(6,7,11)\rangle.\nonumber
\end{align}

Since \(|\mathcal{N}^{(7,1,3,9)}|\geq 12\), the number of cosets in
\([t_7t_1t_3t_9]\) is at most 14. The orbits of \(N^{(7,1,3,9)}\) are:
\(\{\{3\},\{10\},\{1,2,4,...,9,11,..., 14\}\}\). We choose the orbit
representatives: 3, 9, 10.

\begin{enumerate}
	\item[\(t_3\):] We have: 
		\begin{align*}
			t_7t_1(t_3t_9t_3)&\sim_{\ref{rel:alpha}} t_7t_1\alpha_{39}t_3t_9\sim
			t_7t_8t_3t_9 \\
			&\sim_{\ref{rel:stab71}} t_7t_1t_3t_9\in [t_7t_1t_3t_9].
		\end{align*}
	\item[\(t_9\):] It goes back: \(t_7t_1t_3t_9t_9 = t_7t_1t_3\in[t_7t_1t_3]\).
	\item[\(t_{10}\):] We have \(t_7t_1t_3t_9t_{10}\sim
		t_7t_8t_{10}t_6\in[t_7t_1t_2t_4]\).
		\begin{align*}
			(t_7t_1t_3t_9)t_{10} &\sim_{\ref{rel:stab7139}}t_5t_2(t_3t_{13}t_{10}) \sim_{\ref{rel:beta}}
			t_5t_2\beta_{3,13,10}t_3t_6 \\
			&\sim t_{12}t_8t_3t_6 \sim_{\ref{rel:stab712}} t_7t_8t_{10}t_6\in [t_7t_1t_2t_4].
		\end{align*}
		The last \(\sim\) follows by using the relation \(t_{12}t_8t_3\sim
		t_7t_8t_{10}\in[t_7t_1t_2]\). 
\end{enumerate}
This completes the coset \([t_7t_1t_3t_9]\). There are no new cosets.

\subsubsection{\([t_7t_1t_2t_4]\)}
\label{sssec:7124}

The double coset \([t_7t_1t_2t_4]\) has the relations \(t_7t_1t_2t_4\sim
t_3t_7t_2t_{12}\) and \(t_7t_1t_2t_4 \sim t_4t_{10}t_2t_{13}\). These are proved
in an analagous way to the relations in \([t_7t_1t_3t_9]\). The coset stabilizer
expands to:

\begin{align}
	\label{rel:stab7124}
	\mathcal{N}^{(7,1,2,4)} &\geq \langle (1,13)(3,7)(4,11)(5,12)(6,8)(10,14), \\
	&(1,10,12)(3,5,8)(4,13,7)(6,14,11)\rangle.\nonumber
\end{align}

Since \(|\mathcal{N}^{(7124)}|\geq 12\), the number of cosets in \([t_7t_1t_2t_4]\)
is at most 14.  The orbits of \(\mathcal{N}^{(7124)}\) are: \(\{\{2\}, \{9\},
\{1,3,...,8,10,..., 14\}\}\). We choose the orbit representatives: 2, 4, 9.

\begin{enumerate}
	\item[\(t_2\):] We have 
		\begin{align*}
			(t_7t_1t_2t_4)t_2 & \sim_{\ref{rel:stab7124}}t_{13}t_{12}(t_2t_7t_2) \sim_{\ref{rel:alpha}}
			t_{13}t_{12}\alpha_{2,7}t_2t_7 \\
			&\sim t_5t_6t_2t_7 \sim_{\ref{rel:stab7}} (t_{12}t_6t_2)t_7\\
			&\sim_{\ref{rel:stab713}} t_3t_6t_9t_7\in[t_7t_1t_3t_9].
		\end{align*}
	\item[\(t_4\):] It goes back \(t_7t_1t_2t_4t_4 = t_7t_1t_2\in[t_7t_1t_2]\).
	\item[\(t_9\):] We have
		\begin{align*}
			t_7t_1(t_2t_4t_9) &\sim_{\ref{rel:beta}} t_7t_1 \beta_{2,4,9}t_2t_{11}\sim
			t_{10}t_6t_2t_{11}\\
			&\sim_{\ref{rel:stab71}} t_3t_{13}t_2t_{11}\in[t_7t_1t_2t_4].
		\end{align*}
\end{enumerate}

As there are no new cosets, we have completed the double coset enumeration
process.

\subsection{Cayley graph and coset table.}
\label{ssec:cayley-table}

We can represent the work in this section with a collapsed Cayley graph, see Figure
\ref{fig:cayley-graph}.  The circles represent double cosets and lines represent
multiplication by \(t_i\)s.  The numbers inside of the circles represent the
number of single cosets within the double coset, while the numbers on the
outside of the circles indicate the number of \(t_i\)s going to the next double
coset.

\begin{center}
	\begin{table}[h]
		\caption{Double Cosets \([\omega] = \mathcal{M}\omega\mathcal{N}\)}
	\begin{tabular}{lll}\hline\noalign{\vskip 0.05in}
		Label \([\omega]\) & Coset Stabilizing Subgroup \(\mathcal{N}^{(\omega)}\) & Number \\
		& & of cosets \\ \hline
		\([\star]\) & \(\mathcal{N}\) & 1 \\
		\([t_7]\) & \(\mathcal{N}^{(7)}\cong \Sigma_4\) & 7 \\
		\([t_7t_1]\) & \(\mathcal{N}^{(7,1)}\cong \mathbb{Z}_2\times\mathbb{Z}_2\) & 42 \\
		\([t_7t_1t_2]\) & \(\mathcal{N}^{(7,1,2)}\cong \mathbb{Z}_2\) & 84 \\
		\([t_7t_1t_3]\) & \(\mathcal{N}^{(7,1,3)}\cong \mathbb{Z}_2\) & 84\\
		\([t_7t_1t_2t_3]\) & \(\mathcal{N}^{(7,1,2,3)}\cong \mathbb{Z}_2\) & 84 \\
		\([t_7t_1t_2t_4]\) & \(\mathcal{N}^{(7,1,2,4)}\cong A_4\) & 14\\
		\([t_7t_1t_3t_9]\) & \(\mathcal{N}^{(7,1,3,9)}\cong A_4\)& 14\\ \hline
	\end{tabular}
	\label{tab:cosets}
	\end{table}
\end{center}

\begin{figure}[h]
\tikzstyle{dcoset}=[circle,fill=white,draw=black,text=black, minimum size=10 mm]
\begin{tikzpicture}
	\node[dcoset] (*)	[label=below:{[\(\star\)]}] 
				[label={5:14}] {1};
	\node[dcoset] (7) 	[right=1.5cm of *] [label=below:{[\(t_7\)]}] 
				[label={175:2}]
				[label={5:12}] {7};
	\node[dcoset] (71) 	[right=1.5cm of 7]  [label=below:{[\(t_7t_1\)]}]
				[label={175:2}]
				[label={45:4}]
				[label={-30:4}] {42};
	\node[dcoset] (712) [above right=1cm and 1.5 cm of 71]
	[label=above:{[\(t_7t_1t_2\)]}]
				[label={196:2}] 
				[label={245:1}]
				[label={5:2}]
				[label={-40:6}]{84};
	\node[dcoset] (713) [below right=1cm and 1.5cm of 71]  [label=below:{[\(t_7t_1t_3\)]}] 
				[label={154:2}]
				[label={115:1}]
				[label={40:6}]
				[label={-5:2}]{84};
	\node[dcoset] (7123) [right=2.75cm of 71] [label=left:{[\(t_7t_1t_2t_3\)]}] 
				[label={92:6}]
				[label={-92:6}]{84};
	\node[dcoset] (7139) [right=2.25cm of 713] [label=below:{[\(t_7t_1t_3t_9\)]}]
				[label={-175:12}]
				[label={95:1}] {14};
	\node[dcoset] (7124) [right=2.25cm of 712] [label=above:{[\(t_7t_1t_2t_4\)]}]  
				[label={175:12}]
				[label={-95:1}]{14};
	
	\draw[-] (*) -- (7);
	\draw[-] (7) -- (71);
	\draw[-] (71) to [bend left=30] (712);
	\draw[-] (71) to [bend right=20] (713);
	\draw[-] (712) to [bend right=20] (713);
	\draw[-] (713) to (7139);
	\draw[-] (712) to (7124);
	\draw[-] (7139) to (7124);
	\draw[-] (712) to [bend left=20] (7123);
	\draw[-] (713) to [bend right=20] (7123);
	\draw[-] (71) to [in=140, out=110,loop] node[above=1pt] {2+2} (71);
	\draw[-] (712) to [in=160, out=130, loop] node[above left=0pt] {1+2} (712);
	\draw[-] (713) to [in=210, out=180, loop] node[below left] {1+2} (713);
	\draw[-] (7123) to [in=15, out=-15,loop] node[above=2pt] {1+1} (7123);
	\draw[-] (7124) to [in=15, out=-15,loop] node[above=2pt] {1} (7124);
	\draw[-] (7139) to [in=15, out=-15,loop] node[above=2pt] {1} (7139);
\end{tikzpicture}
\caption{Collapsed Cayley Graph of \(\mathcal{G}\) Over \(\mathcal{M}\)}
\label{fig:cayley-graph}
\end{figure}
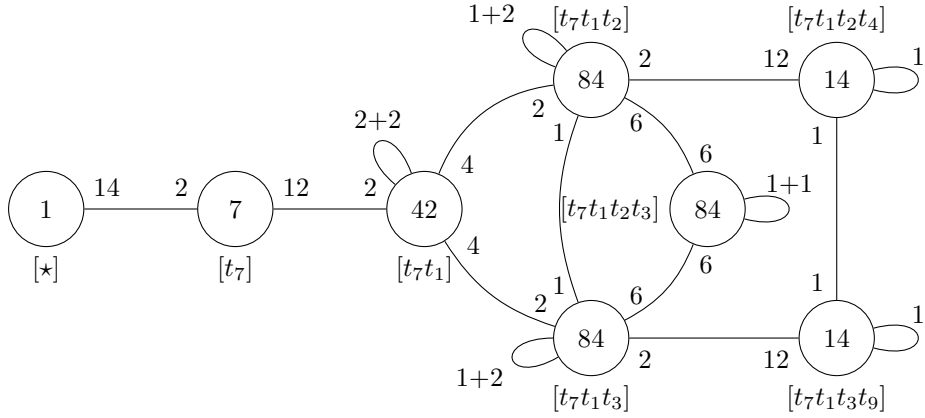

\section{Main Result}
\label{sec:main-result}

\subsection{\(\mathcal{G}\cong M_{22}\).}
\label{ssec:main-proof}

We use Iwasawa's Lemma and the transitive action of \(\mathcal{G}\) on the set
of single cosets \(\Omega:=\{\mathcal{M}\omega\}\) to prove \(\mathcal{G}\) is simple of
order 443,520.  By \cite{ATLAS} and \cite{Par1} there is only one simple group
of order 443,520, the Mathieu Group \(M_{22}\).  We will conclude that
\(\mathcal{G}\cong M_{22}\).

\begin{lemma}
	The order of \(\mathcal{G}\) is 443,520.  Furthermore, \(\mathcal{G}\) acts
	faithfully on \(\Omega=\mathcal{G/M}\).
	\label{lem:order-faithful}
\end{lemma}

\begin{proof}
	Since \(\Omega\) is a transitive \(\mathcal{G}\)-set of degree \(330\),
	\(|\mathcal{G}|\leq 330|\mathcal{G^M}|\), where \(\mathcal{G^M}\) is the
	stabilizer of the single coset \(M\). The inequality is due to the fact there
	could be additional collapsing that we are unaware of. But \(\mathcal{M}\) is
	only stabilized by elements of \(\mathcal{M}\).  Hence \(\mathcal{G^M=M}\) and
	\(|\mathcal{G^M}|=|\mathcal{M}|=1344\). We conclude \(|\mathcal{G}|\leq
	443,520\).  By Corollary \ref{cor:main-order} we have \(|\mathcal{G}|\geq
	443520\) and so equality follows.
\end{proof}

\begin{lemma}
	\(\mathcal{G}\) is perfect.
	\label{lem:G-perfect}
\end{lemma}

\begin{proof}
	Let \(\mathcal{G}'\) denote the commutator subgroup. Since
	\(\mathcal{G}=\langle \mathcal{N},t\rangle\) and \(\mathcal{N}\) is simple, we
	have \(\mathcal{N\leq G'}\). It therefore is sufficient to show
	\(t_7\in\mathcal{G}'\). Applying Relation (\ref{rel:alpha}) we get
	\(t=t_7\sim \alpha_{1,7}[t_7,t_1]\in\mathcal{G}'\) and we are finished.
\end{proof}

\begin{lemma}
	The point stabilizer, \(\mathcal{G^M\leq G}\), possesses a normal Abelian subgroup
	\(\mathcal{K}\) whose conjugates generate \(\mathcal{G}\).
	\label{lem:abelian-stabilizer}
\end{lemma}

\begin{proof}
	Let \(\mathcal{K}=\langle s_7,s_1,s_2\rangle\leq \mathcal{G^M=M}\) and let
	\(\bar{\mathcal{K}}\) be the normal closure of \(\mathcal{K}\) in
	\(\mathcal{G}\). 
	
	Using Relation (\ref{rel:delta}), we have
	\(t_1t_8=\delta_{5,1}t_5t_1t_8t_{12}\).  Using Relation (\ref{rel:delta})
	\[
		s_1^{t_5}=t_5(t_1t_8)t_5=t_5\delta_{5,1}t_5t_1t_8t_{12}t_5 =
		\delta_{5,1}s_1s_5=\delta_{5,1}s_7\in \bar{\mathcal{K}}.
	\] 
	Since \(s_7s_1^{t_5}\in\bar{\mathcal{K}}\) we conclude
	\(\delta_{5,1}\in\bar{\mathcal{K}}\). Since \(\delta_{5,1}\) is an
	involution and all involutions are conjugate in \(\mathcal{N}\), we know
	\(y\in\mathcal{N}\).

	Conjugating the relation \(y=t_1t_{12}t_1t_{12}t_1\) by \(t_1t_{12}\), yields
	\(t_1\in \bar{\mathcal{K}}\). Since the conjugacy class of \(y\) is
	transitive, we have \(t_1,...,t_{14}\in \bar{\mathcal{K}}\). Hence,
	\(x=t_7t_6t_5t_4t_3t_2t_1t_7\in \bar{\mathcal{K}}\). Finally, we have
	\(\bar{\mathcal{K}}=\langle x,y,t_1\rangle=\mathcal{G}\).
\end{proof}

\begin{lemma}
	The group \(\mathcal{G}\) acts primitively on \(\Omega\).
	\label{lem:G-prim}
\end{lemma}

\begin{proof}
	Let \(\mathcal{B}\) be a nontrivial block. The action of \(\mathcal{G}\) on
	\(\Omega\) is transitive so we may assume \(\mathcal{M\in B}\). The action of
	\(\mathcal{N}\) on \(\mathcal{B}\) therefore preserves \(\mathcal{B}\) as
	\(\mathcal{MN} = \mathcal{M}\).

	Since \(\mathcal{B}\) is nontrivial, there exists a word \(\omega\) in the
	\(t_i\)s such that \(\mathcal{M}\omega\neq \mathcal{M}\). There are two
	possibilities for such an \(\omega\). It is either in \([t_7]\) or it is not. If
	it is in \([t_7]\), then \(\mathcal{B}\) is stabilized by all of the \(t_i\)s
	and is therefore trivial.

	Suppose \(\omega\) is not in \([t_7]\), then the action of \(\mathcal{N}\)
	preserves \(\mathcal{B}\), so we get the entire double coset
	\(\omega\in[\omega]\subset\mathcal{B}\). Each of the double cosets, distinct
	from \([\star],[t_7]\) are stabilized by a \(t_i\). It follows that
	\(\mathcal{B}\) is stabilized by right multiplication by \(t_i\) for all
	\(i\). Hence, \(\mathcal{B} = \Omega\) and so \(\mathcal{B}\) is trivial.
\end{proof}

At last, we apply Iwasawa's Lemma:

\begin{proof}[Proof of Main Theorem]
	Lemmas \ref{lem:order-faithful}, \ref{lem:G-perfect},
	\ref{lem:abelian-stabilizer}, \ref{lem:G-prim}, and Iwasawa's lemma imply
	\(\mathcal{G}\) is simple.
	
	Since \(|\mathcal{G}|=443,520\), we have \(\mathcal{G}\cong M_{22}\) by
	\cite{ATLAS}, and \cite{Par1}. 
\end{proof}

\section{Maximal Subgroups of \(M_{22}\)}
\label{sec:comparison}

\subsection{The presentation.}
\label{ssec:comparison}

The presentation we have constructed for \(M_{22}\) is
\begin{align*}
	M_{22} = \langle x,y,t &\mid x^7, y^2, (xy)^3, [x,y]^4, \\
	&t^2, [t^{x^2},yx^{-1}], [t,y], \\
	&(yt^{x^2})^5, (xyx^2t^x)^5, (xt)^8\rangle.
\end{align*}
where the first line of relations corresponds to the control group, \(L_3(2)\),
the first and second corresopnd to the progenitor \(2^{\star 14}:L_3(2)\), and
finally all three define \(M_{22}\). We will identify our presentation with the
one from Remark \ref{rem:main} using the permutations therein.

\subsection{Maximal subgroups.}
\label{sec:maximal-subs}

The maximal subgroups of \(M_{22}\) are classified. There are 8 conjugacy
classes of subgroups inside \(M_{22}\). We have the following representatives of
conjugacy classes of maximal subgroups in terms of our presentation:

\begin{align*}
	\langle x, t^{x^6t}\rangle &\cong L_3(4); \\
	\langle xt, y^{x^3}\rangle &\cong 2^4:A_6; \\
	\langle x,t^{xtx^4t}\rangle &\cong A_7; \\
	\langle x, t^{x^3tx^6t}\rangle &\cong A_7; \\
	\langle xt,t^{xyx^4}\rangle &\cong 2^4:S_5; \\
	\langle x,y,tt^{x^6 yx} \rangle &\cong 2^3:L_3(2); \\
	\langle x^2t, y^{x^4}\rangle &\cong M_{10}; \\
	\langle xy,t\rangle &\cong L_2(11).
\end{align*}

\begin{table}[h]
	\caption{The MOG Labeling.}
	\begin{tabular}{|cc|cc|cc|} \hline
		24 & 14 & 17 & 11 & 22 & 19 \\
		23 & 8 & 4 & 13 & 1 & 9 \\ \hline 
		3 & 20 & 16 & 7 & 12 & 5 \\
		15 & 18 & 10 & 2 & 21 & 6 \\ \hline
	\end{tabular}
	\label{table:mog}
\end{table}

Indeed, we can use the MOG array as described in \cite{curt-mog} with the
labelling in Table \ref{table:mog} to see these isomorphisms. The following can
be checked by hand using the permutations in Remark

\ref{rem:main}:
\begin{enumerate}
	\item[\(L_3(4)\):] the stabilizer of \(\{4\}\);
	\item[\(2^4:A_6\):] the stabilizer of the hexad \(\{5,8,9,10,17,18\}\);
	\item[\(A_7\):] the stabilizer of the heptad \(\{1,8,12,14,15,17,20\}\);
	\item[\(A_7\):] the stabilizer of the heptad \(\{5,6,11,16,19,20,21\}\);
	\item[\(2^4:S_5\):] the stabilizer of the pair \(\{5,18\}\);
	\item[\(2^3:L_3(2)\):] was proven earlier; however, it is the 
		stabilizer of the octad \(\{2,3,4,7,9,13,18,22\}\);
	\item[\(M_{10}\):] the stabilizer of the dodecad
		\(\{2,4,5,11,12,16,18,19,20,22\}\);
	\item[\(L_2(11)\):] the stabilizer of the endecad
		\(\{1,2,5,7,8,9,10,14,15,17,22\}\).
\end{enumerate}

\subsection{Covering groups as homomorphic images}

The progenitor \(2^{\star 14}:L_3(2)\) has at least two other important
homomorphic images related to \(M_{22}\). We have
\begin{align*}
	2\cdot M_{22} &\cong \frac{2^{\star 14}:L_3(2)}{(xyx^2t^x)^5,(xt)^8}; \\
	3\cdot M_{22} &\cong \frac{2^{\star 14}:L_3(2)}{(xyx^2t^x)^5,(yt^{x^2})^5}.
\end{align*}
These isomorphisms can be checked with a computer algebra software. It would be
interesting to see if the other two covers, \(6\cdot M_{22}\) and \(12\cdot
M_{22}\), appear as homomorphic images of this progenitor as well.

\bibliographystyle{alpha}
\bibliography{mybib}
\end{document}